\documentclass[preprint,11pt]{elsarticle}

\usepackage{amssymb, amsmath, amsthm, amsfonts}
\usepackage{extsizes}
\usepackage{amssymb}
 \usepackage{amsmath}
 \usepackage{amsthm}
 
  \usepackage{xcolor}
  
\newtheorem{theorem}{Theorem}[section]

\newtheorem{proposition}[theorem]{Proposition}

\theoremstyle{definition}

\def\r{\mathbb R}
\def\l{\mathbb L}

\begin{document}

\begin{frontmatter}
 
\title{The Newton's problem assuming non-constant density of the fluid}
 
\author{Rafael L\'opez} 
\ead{rcamino@ugr.es}
\address{Department of Geometry and Topology. University of Granada. 18071 Granada, Spain}

\begin{abstract}
This paper investigates the  Newton's problem of minimal resistance for a body moving through a fluid whose density decreases exponentially with altitude.   We prove the local existence and regularity of radial solutions $u(r)$ satisfying the initial conditions $u(0)=u'(0)=0$ using a fixed-point theorem. We show that the maximal domain of the solution is finite, $[0, r_M)$,  terminating at a critical slope $u'(r_M) = \frac{1}{\sqrt{3}}$.  
\end{abstract}
 

\end{frontmatter}

\section{A variant of Newton's problem}

Newton proposed in his {\it Principia Mathematica} the problem of minimal resistance: find the shape of a solid body which offers minimal resistance while moving with constant velocity through a homogeneous fluid. However, in the case where the fluid is   air, this hypothesis is not realistic because it is known that the density of the Earth's atmosphere decreases exponentially   as altitude increases. This is due to the fact that both pressure and temperature generally decrease, leading to a rapid reduction in air density. If the $z$-coordinate indicates the height with respect to a reference plane, the density $\rho(z)$ of the air follows the law $\rho(z)=\rho_0e^{-c z}$, where $c>0$ is a scale height and $\rho_0$ is the atmospheric density at sea level $z=0$ ($\rho_0 = 1.225~\mathrm{kg/m^3}$) \cite{an,ah,wa}.  This stratification is a key factor in the design of aerospace vehicles, as the resistance encountered by a body changes fundamentally as it traverses different atmospheric layers. In the literature, other variants of Newton's problem have been studied by assuming that the distribution of the velocity of the particles is centrally symmetric \cite{pt} or   thermal motion of medium particles \cite[Ch. 3]{pl}.

In this paper, we will assume this hypothesis on the exponential decay of the density of the air along one spatial  coordinate and that the body moves along this coordinate. We will derive  the  resistance of the body following the standard approach of the classical model. We refer the reader to the surveys of the problem \cite{bu,bk} and the references therein. Consider Cartesian coordinates $(x,y,z)$ in Euclidean space $\r^3$ and suppose that a body $\mathcal{B}$ moves along the $z$-coordinate. After rescaling, we assume that as $\mathcal{B}$ moves up, the density decreases as $e^{-z}$. Thus, the velocity of the particles is $\vec{v}_i=\vec{u}$, where $\vec{u}=(0,0,-1)$ and we identify the mass $m_i$ of a particle with its density $e^{-z}$. When the particles collide with the surface $S$ of the body  $\mathcal{B}$, they transfer its normal momentum, assuming there is no friction between the fluid and   $S$. At the impact point of the particle with $S$, the model satisfies the sine-squared pressure law. Suppose that $S$ is the graph of a piecewise smooth function $u\colon\Omega\to\r$ defined on a domain $\Omega\subset\r^2$. The  initial value of momentum is $\vec{m}_i=m_i\vec{v}_i$. Then there is an angle  $\varphi$ such that 
$\langle\vec{u},N\rangle=\sin\varphi$, where $N$ is the unit normal vector to $S$. Decompose $\vec{v}_i$ into its tangential and normal parts on $S$, 
 $\vec{v}_i=|\vec{v}_i|(\cos\varphi T-\sin\varphi N)$. When a particle collides with $S$, the final velocity $\vec{v}_f$ is  $\vec{v}_f=|\vec{v}_i|(\cos\varphi T+\sin\varphi N)$. The resistance due to the particle is given by  
 $$m_i\langle \vec{v}_f-\vec{v}_i,\vec{v}_i\rangle=2m_i\sin\varphi \langle N,\vec{v}_i\rangle=2m_i(\sin\varphi)^2=2e^{-u}(\sin\varphi)^2.$$
 Since  $N=(-Du,1)/\sqrt{1+|Du|^2}$, then $(\sin\varphi)^2=1/(1+|Du|^2)$. 
 Assuming also that the particles after hitting against $S$ do not touch   $S$ again, the  resistance   is given by the integral
 \begin{equation}\label{eq1}
 E[u]=\int_\Omega \frac{e^{-u}}{1+|Du|^2}\, dxdy.
 \end{equation}
 Although the integrand in \eqref{eq1} is non-convex with respect to the gradient -- a feature shared with the classical Newton functional -- the presence of the exponential weight $\rho(u)=e^{-u}$ introduces a significant structural change in the associated Euler-Lagrange equation. It is straightforward that the Euler-Lagrange equation of $E$ is 
 \begin{equation}\label{q0}
2(1+|Du|^2)\Delta u 
- 8 D^2u(Du,Du)
= (1+|Du|^2)(1+3|Du|^2).
\end{equation}
 Recall that in the classical model, the equation is  $(1+|Du|^2)\Delta u-4 D^2u(Du,Du)=0$.

 In this paper, we focus on the radial solutions of \eqref{q0}. We will prove that there are radial solutions of \eqref{q0} that intersect the rotation axis orthogonally (Thm.  \ref{t1}). This makes a significant difference with respect to the classical model, where the radial solutions cannot reach the rotation axis orthogonally unless in the trivial case where $u$ is a constant function. It is important to emphasize that our study focuses on the existence and regularity of classical $C^2$ extremals rather than the determination of a global minimizer for the functional $E[u]$. In the classical Newton problem, it is well known that the absolute minimizer lacks $C^2$ regularity, typically requiring a jump in the derivative (such as a frontal flat disk) to achieve minimum resistance.  By establishing these smooth radial solutions, we show that the exponential decay of density allows for regular geometries at the axis of symmetry that are impossible in the classical homogeneous model. Specifically, we prove that these smooth radial solutions exhibit a monotone profile where the slope $u'$ increases from zero at the axis of symmetry until it reaches a critical value $1/\sqrt{3}$ (Thm. \ref{t2}).

  On the other hand, we also study the maximal domain of the radial solutions, proving that this domain is a bounded interval   where the slope of the solution at the end point of the interval is finite (Thm.  \ref{t2}).  As in the   classical problem, the unconstrained minimum resistance problem has no solution in the class of functions $\mathcal{C}(B_R)=\{u\in C^1( \overline{B}_R(0))\colon u\geq 0\}$. For this, consider in polar coordinates the cone $u_\lambda(r)=\lambda R-\lambda r$, $\lambda>0$. Then  $u_{\partial B_R(0)}=0$, $u\geq 0$ and   $ \lim_{\lambda\to \infty}E[u_\lambda]= 0$.

\section{The radial case}\label{s2}

Suppose that $S$ is a surface of revolution about the $z$-axis and that its generating curve is $\gamma(r)=(r,0, u(r))$, where 
$r\in (a,b)$ for some $0\leq a<b$.  The Euler-Lagrange equation \eqref{q0} becomes
 \begin{equation}\label{ur}
 \frac{d}{dr}\left(\frac{r u'}{(1 + u'^2)^2}\right) = \frac{r (1+3u'^2)}{2(1 + u'^2)^2},
\end{equation}
or in expanded form, 
\begin{equation}\label{ur2}
r \left( 1 + 4 u'^2 + 3 u'^4 \right) - 2 u' \left( 1 + u'^2 \right) + 2 r u'' \left( 3 u'^2 - 1 \right) = 0.
\end{equation}
 In the classical model, the right hand-side of \eqref{ur} is $0$ which implies that either $u$ is constant or the solution   is defined away from $0$ \cite{go}.  
 
In our model of non-constant density, it is not possible to obtain a first integral of \eqref{ur}.  In fact,  a constant function is not a solution of \eqref{ur}. This raises the possibility that the function $u$ can intersect the $u$-axis orthogonally, which cannot occur in the classical model. In terms of \eqref{ur2}, the initial conditions  are $u(0)=u_0$ and $u'(0)=0$. However, equation \eqref{ur2} is singular at $r=0$, so classical ODE techniques do not guarantee the existence of solutions. We prove the existence using a fixed point argument. Since if $u(r)$ is a solution of \eqref{ur2}, then $u(r)+\lambda$ is also a solution, we can assume without loss of generality that the initial condition at $r=0$ is  $u(0)=0$.

\begin{theorem}\label{t1}
The initial value problem \eqref{ur2} with initial conditions $u(0)=0$ and $u'(0)=0$ has a solution $u\in C^2([0,R])$ for some $R>0$. 
\end{theorem}

\begin{proof} Define 
$$f(x)=\frac{x}{(1+x^2)^2} ,\qquad g(x)=\frac{1+3x^2}{2(1+x^2)^2},$$
with $f:(-m_0,m_0)\to (-m_1,m_1)$, where $m_0=\frac{1}{\sqrt{3}}$ and $m_1= \frac{3\sqrt{3}}{16}$, and $g:\r\to\r$. The function $f$ is one-to-one and   increasing. A function $u\in C^2([0,R])$ is a solution of our initial value problem if and only if
 $$u(r)= \int_0^rf^{-1}\left(\frac{1}{s}\int_0^s tg(u'(t))\, dt\right)\, ds.$$
 Fix $R>0$, to be chosen later. For $u\in C^1([0,R])$, define the operator 
$$
( \mathsf{T} u)(r)=\int_0^r f^{-1}\left( \int_0^s\frac{t}{s}g(u'(t))\, dt\right)\, ds.
$$
It is clear that a fixed point  $\mathsf{T}u=u$ solves our initial value problem. We apply the Banach fixed point theorem in a closed ball $\overline{\mathcal{B}(0,\epsilon)}\subset C^1([0,R])$ around the function $u=0$ for suitable small $\epsilon>0$. Here, $C^1([0,R])$ is endowed with the norm $\|u\|=\|u\|_\infty+\|u'\|_\infty$. Let us fix $\epsilon>0$ with $\epsilon<m_0$. Since $f^{-1}$ is increasing  and $f^{-1}(0)=0$, let $m_2>0$ be such that $f^{-1}(m_2)=\frac{\epsilon}{2}$. We restrict $f$ and $g$ to   $[-\frac{\epsilon}{2},\frac{\epsilon}{2}]$. Moreover, the range of $g$ is $(0,\frac{9}{16}]$, in particular, $|g|<1$

As a first choice, take $R<2m_2 $, although  $R$ will be modified later. Note that $R<1$. Then $ \mathsf{T}$ is well-defined. Indeed, since $|g|<1$, we have 
$ \left|\int_0^s\frac{t}{s}g(u'(t))\, dt\right|\leq \frac{s}{2}\leq \frac{R}{2}<m_2$. Then we can apply $f^{-1}$.  We perform the proof in two steps.

\begin{enumerate}
\item We prove that $ \mathsf{T}(\overline{\mathcal{B}(0,\epsilon)})\subset \overline{\mathcal{B}(0,\epsilon)}$. Since $ f^{-1}$ is increasing, $|g|<1$ and $R<1$, we have

\begin{equation*}
|( \mathsf{T} u)(r)| \leq \int_0^r f^{-1}\left(\int_0^s\frac{t}{s}\, dt\right)\, ds 
\leq \int_0^r f^{-1}\left(\frac{s}{2}\right)\, ds \leq \frac{\epsilon}{2}r\leq\frac{\epsilon}{2}.
\end{equation*}
 Thus $\| \mathsf{T} u\|_\infty\leq\frac{\epsilon}{2}$. Similarly, we have
\begin{equation*}
 |( \mathsf{T} u)'|\leq f^{-1}\left(\left|\int_0^r\frac{t}{r}g(u'(t))\, dt\right|\right)\leq \left| f^{-1} (\frac{r}{2} )\right| \leq\frac{\epsilon}{2}.
\end{equation*}
 Then $\|( \mathsf{T} u)'\|_\infty\leq \frac{\epsilon}{2}$, proving  that $\| \mathsf{T} u \|\leq \epsilon$. 
\item We prove that $ \mathsf{T}\colon \overline{\mathcal{B}(0,\epsilon)}\to \overline{\mathcal{B}(0,\epsilon)}$ is a contraction.
 Let $L_{ f^{-1}}$ and $L_g$ be the Lipschitz constants of $ f^{-1}$ and $g$ in  $[-\frac{\epsilon}{2},\frac{\epsilon}{2}]$.  Given $u,w\in \overline{\mathcal{B}(0,\epsilon)}$ and $r\in [0,R]$, we have
\begin{equation}\label{c1}
\begin{split}
|( \mathsf{T} u)(r)-( \mathsf{T}w)(r)|&\leq L_{ f^{-1}} \left|\int_0^r \left(\int_0^s\frac{t}{s}(g(u')-g(w'))\, dt\right) \, ds\right|\\
 & \leq L_{ f^{-1}}L_g \|u'-w'\|_\infty\int_0^r(\int_0^s\frac{t}{s}\, dt)\, ds\\
 &=\frac14 L_{ f^{-1}}L_g  R^2  \|u'-w'\|_\infty\leq \frac14 L_{ f^{-1}}L_g  R^2  \|u-w\|.
\end{split}
\end{equation}
Analogously, 
\begin{equation}\label{c2}
\begin{split}
|( \mathsf{T} u)'(r)-( \mathsf{T}w)'(r)| &\leq \left| f^{-1}\left(\int_0^r \frac{t}{r}(g(u')-g(w'))\, dt\right)\right|\\
&\leq L_{ f^{-1}}L_g \|u'-w'\|_\infty \int_0^r \frac{t}{r}\, dt= L_{ f^{-1}}L_g \frac{r}{2} \|u'-w'\|_\infty \\
&\leq \frac12 L_{ f^{-1}}L_g  R \|u'-w'\|_\infty\leq \frac12 L_{ f^{-1}}L_g R\|u-w\|.
 \end{split}
\end{equation}
We need to modify $R$ by taking  
$R\leq \min\{2m_2,\frac{1}{\sqrt{ L_{ f^{-1}}L_g}},\frac{1}{ 2L_{ f^{-1}}L_g}\}.$
From \eqref{c1} and \eqref{c2}, we deduce that $\|\mathsf{T} u-\mathsf{T}w\|_\infty\leq\frac14\|u-w\|$ and $\|(\mathsf{T} u)'-(\mathsf{T}w)'\|_\infty\leq\frac14\|u-w\|$, respectively. This gives $\| Tu-Tw\|\leq \frac12\|u-w\|$, hence $\mathsf{T}$ is a contraction.

\end{enumerate}
 Finally, we write \eqref{ur2} as
 \begin{equation}\label{ur3}
 u''=-\frac{1+4u'^2+3u'^4}{2(3u'^2-1)} +\frac{u'(1+u'^2)}{r (3u'^2-1)}.
 \end{equation}
 The $C^2$-regularity up to $0$ of $u(r)$ is verified by using the  L'H\^{o}pital  rule, because we have $u''(0)=\frac12-\lim_{r\to 0}\frac{u'(r)}{r}=\frac12-u''(0)$, 
hence $u''(0)=\frac14$. 
\end{proof}

  We now show that the maximal domain of the radial solution $u$ is finite. This occurs because at the end of this interval, the coefficient $3u'^2-1$ of $u''$ in \eqref{ur2}   vanishes.

\begin{theorem}\label{t2} The maximal domain of a solution $u(r)$ of \eqref{ur2} with initial conditions $u(0)=u'(0)=0$ is   a bounded interval    $[0,r_M)$, where   $\lim_{r\to r_M}u'(r)=\frac{1}{\sqrt{3}}$. We also have the numerical values $r_M\approx 1.230$ and $u(r_M)\approx 0.228$.
\end{theorem}
\begin{proof}
From Thm.   \ref{t1}, we know $u''(0)=\frac14>0$, implying  that $u$ increases initially around $r=0$, and $u''$ is positive close to $0$. If $r_1>0$ were the first time that $u'$ vanished, then \eqref{ur2} would yield $u''(r_1)=\frac12$, and $r_1$ would be a minimum of $u$, which is not possible. This proves that $u'>0$ in the maximal domain of $u$. 

In order to continue  the study of the behavior of the solution $u(r)$, it is more convenient to use  parametric coordinates   $x(s)=r$ and $y(s)=u(r)$. We reparametrize   $\gamma(s)=(x(s),y(s))$ with   $x'(s)^2+y'(s)^2=1$. Then   there is and angle $\theta=\theta(s)$ such that $x'(s)=\cos\theta(s)$ and $y'(s)=\sin\theta(s)$. Equation  \eqref{ur3} is equivalent to the ODE system
\begin{equation}\label{p1}
\left\{\begin{split}
x'(s)&=\cos\theta(s),\\
y'(s)&=\sin\theta(s),\\
\theta'(s)&=\cos\theta(s)\frac{x(s)(2-\cos(2\theta(s)))-\sin(2\theta(s))}{2x(s)(2\cos(2\theta(s))-1)}.
\end{split}\right.
\end{equation}
The function $\theta'(s)$ is the curvature of the curve $\gamma(s)$.  Since  $y(s)$ does not appear in the expression for $\theta'$, and   $y(s)$ can be recovered from $\theta(s)$, we can study the behavior of the solutions $\gamma(s)$ by considering the autonomous system 
\begin{equation}\label{sys}
\left\{\begin{split}
x'&=\cos\theta,\\
\theta'&=\cos\theta\frac{x(2-\cos(2\theta))-\sin(2\theta)}{2x(2\cos(2\theta)-1)}.
\end{split}\right.
\end{equation}
 Up to $2\pi$-vertical translations in the $\theta$ variable, the phase plane of \eqref{sys} is given by 
$$\Theta=\{(x,\theta)\colon x\in\r,\theta\in(-\frac{\pi}{6},\frac{\pi}{6})\cup (\frac{\pi}{6},\pi+\frac{\pi}{6})\}.$$
 Each solution $\gamma(s)$ has associated with it unique orbit $\alpha_\gamma(s)=(x(s),\theta(s))$ of \eqref{sys}. Standard ODE theory ensures that through every point of $\Theta$ passes a unique orbit, and distinct orbits do not intersect. Let us observe that the equilibrium points are the lines $\theta= \frac{\pi}{2}+k\pi$. See Fig. \ref{fig1}. However, the flow meets these lines with a constant slope  $-\frac12$, because both equations in \eqref{sys} are multiplied by the same factor $\cos\theta$. On the other hand, and except when $x=\pm \frac{1}{\sqrt{3}}$, the flow meets  the lines $\theta=\pm\frac{\pi}{6}$ orthogonally because $\frac{d\theta}{dx}=\infty$. To analyze the behavior of the flow around the points $P_1=(\frac{1}{\sqrt{3}},\frac{\pi}{6})$ and $P_2=(-\frac{1}{\sqrt{3}},-\frac{\pi}{6})$, we   multiply both equations of \eqref{sys}  by $2x(2\cos(2\theta)-1)$, transforming them into  the autonomous system 
 
\begin{equation}\label{sys2}
\left\{
\begin{split}
 x' &= 2x(2\cos(2\theta)-1)\cos\theta, \\ 
\theta '&= \cos\theta (\sin(2\theta)+(\cos(2\theta)-2)x ).
\end{split}\right.
\end{equation}
Now, the points $P_1$ and $P_2$ are isolated equilibrium points in \eqref{sys2}. The linearization of the system at these points shows that   $P_1$ and $P_2$ are centers. Therefore, the flow around these points consists of  topological circles. As a consequence of the phase plane analysis, the flow crosses $\theta=0$ with slope $\frac12$ because $\frac{d\theta}{dx}=\frac12$ at $\theta=0$. Next, the flow moves toward the line $\theta=\frac{\pi}{6}$ being this value its limit, hence   $u'(r)\to \tan(\frac{\pi}{6})=\frac{1}{\sqrt{3}}$. Since $\theta(s)$ is increasing,  the curvature $\theta'$ of  $\gamma$ is positive in its maximal domain. This proves that the radial solution $u(r)$ is convex. 

By inspection of the phase plane, the trajectories in $\Theta$ terminate at the horizontal line $\theta=\frac{\pi}{6}$, with the value of $x(s)$ being finite in the limit. If $[0,a)$ is the maximal domain of the solutions of \eqref{p1}, then   there is $r_M<\infty$ such that $\lim_{s\to a}x(s)=r_M$.  This value $r_M$ defines the maximal domain $[0,r_M)$ of the solution $u(r)$ of \eqref{ur2} with initial conditions $u(0)=u'(0)=0$. A  graph of a solution $u(r)$ appears in   Fig. \ref{fig1}. The values of $r_M$ and $u(r_M)$ have been calculated numerically with Mathematica.
\end{proof}
 
 \begin{figure}[h]
\begin{center}
\includegraphics[width=.4\textwidth]{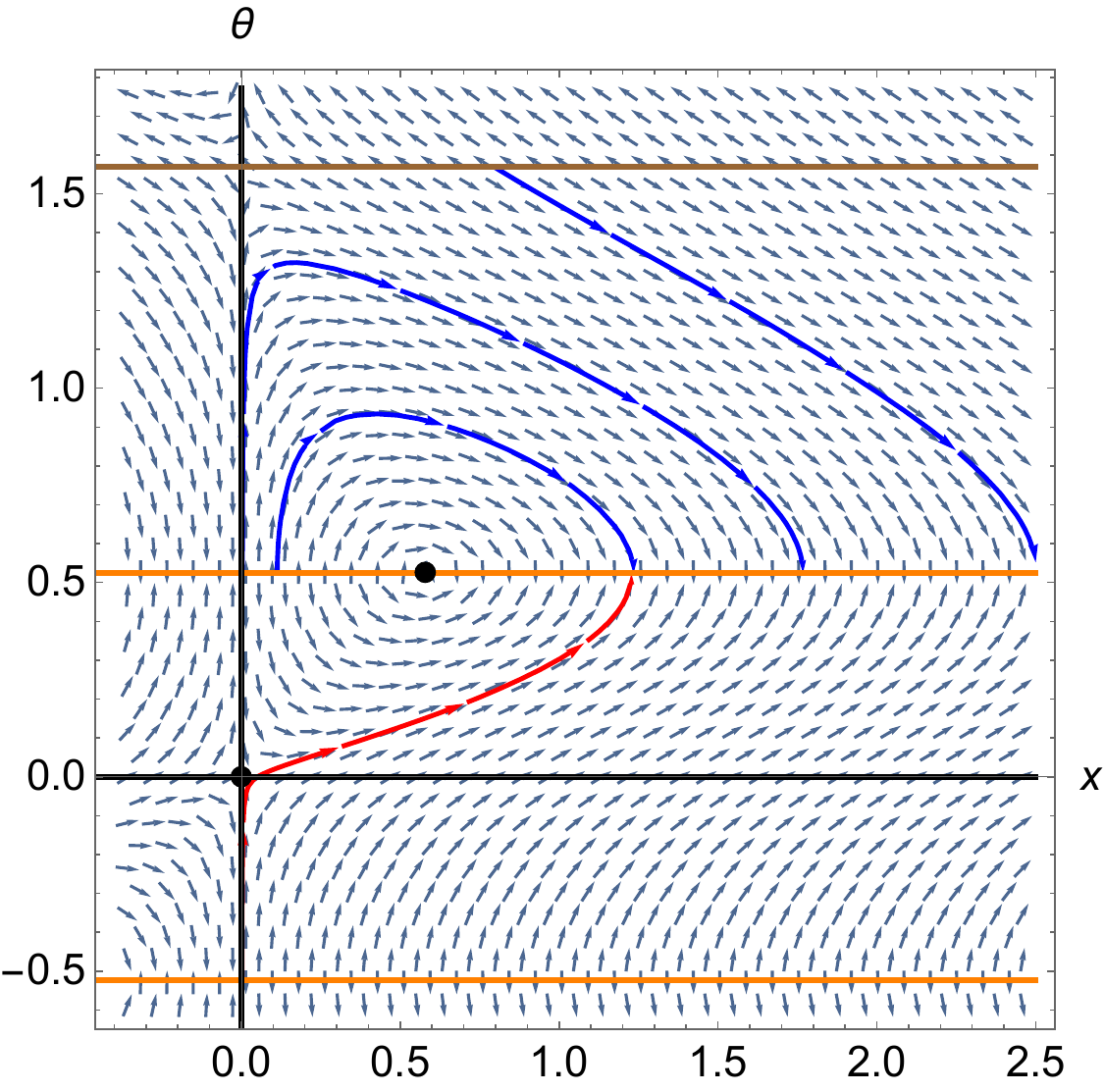} \quad \includegraphics[width=.4\textwidth]{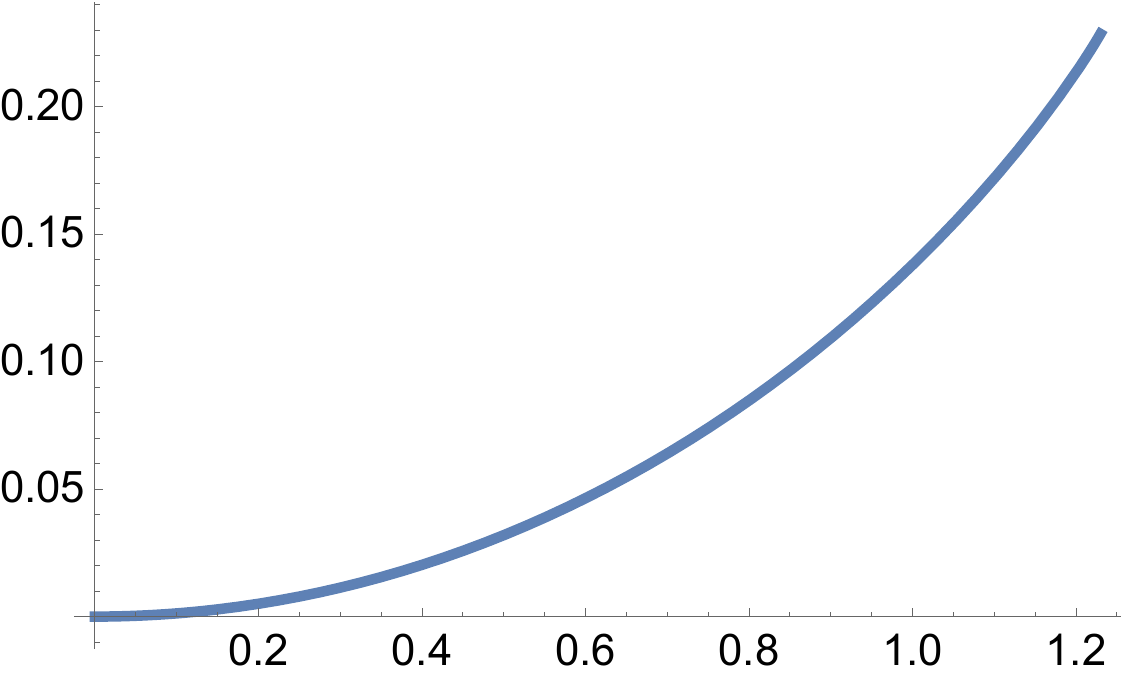}
\end{center}
\caption{Left. The phase plane of \eqref{sys}. The orange lines are $\theta=\pm\frac{\pi}{6}$, the blue line is $\theta=\frac{\pi}{2}$. The red line is a trajectory $\alpha_\gamma$ from the point $(0,0)$ and the blue lines are different trajectories from $(1,\frac{\pi}{2}-\lambda)$ with $\lambda=0.1, 0.5$ and $0.8$.   Right. A radial solution intersecting  orthogonally the rotation axis.}\label{fig1}
\end{figure}

We finish this paper with the following remark.  There exist   solutions   of \eqref{ur} that do not intersect the $y$-axis by taking   $u(r_0)=0$, with $r_0>0$. Around   $(\frac{1}{\sqrt{3}},\frac{\pi}{6})$,  there are trajectories  that  meet $\theta=0$ at two points and   others that meet the line $\theta=\frac{\pi}{2}$ (see blue lines in   Fig. \ref{fig1}).  If $\alpha_\gamma$ meets $\theta=\frac{\pi}{2}$,  the function $u(r)$ is asymptotic to a vertical line.

 \section{Conclusion and outlook}
 
 The present work addressed a physically relevant variation of the classical Newton's problem of minimal resistance, incorporating the exponential decay law of atmospheric density.  Our analysis shows that, unlike the classical model, there exist radial solutions intersecting the rotation axis orthogonally. This property is of particular interest as it allows for the design of streamlined bodies that avoid the structural and stagnation pressure singularities associated with the flat-front designs of the homogeneous case.  These radial solutions  are bounded to a finite domain  terminating when the slope reaches the critical value $ \frac{1}{\sqrt{3}}$. This   study   has direct implications for high-altitude aerodynamics and space vehicle design.  The existence of these $C^2$ extremals provides a theoretical basis for smooth optimal shapes in stratified media.  The fact that the slope of the minimal resistance surface is, up to rescaling, bounded by $\frac{1}{\sqrt{3}}$ at the edge of the domain (Thm.  \ref{t1}) provides a   geometric constraint that should be considered when designing bodies traversing the atmosphere, especially those seeking maximum efficiency in minimizing drag during re-entry or ascent phases.  The existence of $C^2$ extremals meeting the axis orthogonally marks a significant departure from the classical constant-density case. A rigorous proof regarding the optimality of these smooth profiles versus global minimizers represents the next natural step in this research line.


 \section*{Acknowledgements} Rafael L\'opez has been partially supported by MINECO/MICINN/FEDER grant no. PID2023-150727NB-I00, and by the ``Mar\'{\i}a de Maeztu'' Excellence Unit IMAG, reference CEX2020-001105- M, funded by MCINN/AEI/10.13039/ 501100011033/ CEX2020-001105-M.


 \section*{Data availability}  No data was used for the research described in the article.


\begingroup
\small
\setlength{\parskip}{0pt}
\setlength{\itemsep}{0pt}
\setlength{\bibsep}{0pt} 
\bibliographystyle{plain}

\endgroup

\end{document}